\documentclass[11pt]{amsart}
\include{std}
\usepackage{amsmath,amsfonts,amssymb,amscd,verbatim,xypic,latexsym, xy}
\usepackage{graphicx}

\usepackage{psboxit}
\usepackage{ifthen}
\newdimen\minCDarrowwidth
\usepackage{amsmath}
\usepackage{amsfonts}
\usepackage{amssymb}
\usepackage{amscd}
\usepackage{verbatim}

\def\cN{\mathcal{N}}
\def\cQ{\mathcal{Q}}
\def\cP{\mathcal{P}}

\def\cS{\mathcal{S}}

\def\Im{\mbox{ Im }}

 \DeclareMathOperator{\Hom}{Hom}

 \DeclareMathOperator{\Spec}{Spec}
 
\DeclareMathOperator{\Sch}{\mathbf{Sch}}\DeclareMathOperator{\Sets}{\mathbf{Sets}}

\newtheorem{lemma}{Lemma}[section]
\newtheorem{theorem}[lemma]{Theorem}

\newtheorem{proposition}[lemma]{Proposition}

\theoremstyle{definition}
\newtheorem{definition}[lemma]{Definition}
\newtheorem{example}[lemma]{Example}

\newtheorem{remark}[lemma]{Remark}

\newtheorem*{notation}{Notation}

\numberwithin{equation}{section}

\newcommand{\bean}{\begin{eqnarray}}
\newcommand{\eean}{\end{eqnarray}}
\newcommand{\bea}{\begin{eqnarray*}}
\newcommand{\eea}{\end{eqnarray*}}
\newcommand{\be}{\begin{displaymath}}
\newcommand{\ee}{\end{displaymath}}

\newcommand{\ol}{\overline}

\xyoption{arc}

\begin{document}

\title{A universal \'etale lift of a proper local embedding}


\author{Anca~M.~Musta\c{t}\v{a}}
\author{Andrei~Musta\c{t}\v{a}}

\address{School of Mathematical Sciences, University College Cork, Cork, Ireland}
\email{{\tt a.mustata@ucc.ie, andrei.mustata@ucc.ie}}


\date{\today}


\begin{abstract}

To any finite local embedding of Deligne--Mumford stacks $g: Y\to X$
we associate  an \'etale, universally closed morphism $F_{Y/X}\to X$
such that for the complement $Y^2_X$ of the image of the diagonal $Y
\to Y\times_XY$, the stack $F_{Y^2_X/Y}$ admits a canonical closed
embedding in $F_{Y/X}$, and  $F_{Y/X}\times_XY$ is a disjoint union
of copies of $F_{Y^2_X/Y}$.  The stack $F_{Y/X}$ has a natural
functorial presentation, and the morphism $F_{Y/X}\to X$ commutes
with base-change. The image of $Y^2_X$ in $Y$ is the locus of
 points where the morphism $Y \to g(Y)$ is not smooth. Thus for many practical purposes, the
morphism $g$ can be replaced in a canonical way by copies of the
closed embedding $F_{Y^2_X/Y}\to F_{Y/X}$.

\end{abstract}

\maketitle
\bigskip

\section*{Introduction}

    Local embeddings of Deligne-Mumford stacks constitute a natural extension of the notion of closed embeddings of
    schemes. For example, the diagonal of a Deligne-Mumford a stack,
and the natural morphism from its inertia stack to the stack itself, both belong to this class.

  Many difficulties in extending classical algebraic geometry constructions from  the category of
  schemes to stacks stem from the existence of such local embeddings. To solve this problem, one can rely on the local nature for the \'etale topology of these morphisms. Indeed, given a local embedding of algebraic stacks $g:Y \to X$, there exist \'etale atlases $V_0$ and $U$ of $Y$ and $X$ respectively, and a closed embedding $V_0 \hookrightarrow U$ compatible with the morphism $g$. This local construction yields the notions of normal bundle of a local embedding as introduced by A. Vistoli (\cite{vistoli}), and deformation to the normal cone as introduced by A. Kresch (\cite{kresch}), and consequently an intersection theory on smooth Deligne-Mumford stacks.

    In \cite{noi1} we argued that a more refined \'etale presentation of the morphism $g:Y \to X$ is needed if such
    ubiquitous constructions like blow-ups are to be defined for local embeddings. For this purpose, given a proper local
    embedding $g:Y \to X$, we constructed an \'etale atlas $U$ of $X$ such that the fibre product $Y\times_XU$ is a union
    of \'etale atlases $V_i$ of $Y$, each of which is embedded as a closed subscheme in $U$. The locus where the images
    $W_i$ of $V_i$-s intersect pairwise is an \'etale atlas for the stack of non-smooth values of $g$. Moreover, the
    stratification determined by the number of intersecting components $W_i$ indicates how far the morphism $g$ is from
    being \'etale on the image over each point in $g(Y)$. The \'etale atlas $U$ thus encodes essential information about
    the structure of $g$. In \cite{noi1} we set out to translate this information from \'etale atlases to stacks amenable
    to global constructions like e.g. blow-ups, or intersection rings. For a proper $g:Y \to X$, we found  a pair of
    stacks $Y'$ and $X'$ with \'etale, universally closed morphisms $Y'\to Y$ and $X'\to X$, and a morphism $g':Y'\to X'$
    such that $Y'=Y\times_XX'$ is a disjoint union of stacks each embedded as a closed substack of $X'$ via $g'$.
    However, our construction was not unique. Indeed, it depends on the choice of a suitable \'etale atlas $U$ of $X$.

 In this paper we introduce an \'etale, universally closed morphism  $F_{Y/X}\to X$ which is intrinsically
 associated to the proper local embedding $g:Y\to X$, which has
 the desired properties listed above, and which commutes with base
 change. We give a functorial presentation of this stack and study
 its properties in more detail. As applications, this canonical definition provides
 grounds for extending other constructions from schemes to stacks.
 For example, we can now define compactifications of configuration
 spaces for stacks by extending W. Fulton and R. MacPherson's  \cite{fmp} constructions in a coherent, natural way. Also, in our opinion the stack $F_{X/X\times X}$ provides a
 natural context for orbifold products like the ones defined by
 Edidin, Jarvis and Kimura in \cite{ejk} for quotient
 Deligne-Mumford stacks. We will explore such applications in more
 detail in a sequel to this paper.

We start this article by discussing the case when $g: Y\to X$ is a
morphism of Deligne--Mumford stacks which is finite and \'etale on
its image. In \cite{noi1} we showed that such a morphism can be
factored into an \'etale, universally closed morphism $F_{Y/X} \to
X$ and an embedding $Y \hookrightarrow F_{Y/X}$, which identifies
$Y$ with the preimage of $g(Y)$ in $F_{Y/X}$, and such that
$F_{Y/X}\setminus Y \cong X\setminus g(Y)$. In Proposition \ref{lift
for morph etale on its image} we provide a detailed list of
properties for $F_{Y/X}$, some of which will prove very useful in
more general set-up. For example, property (10) will lead to a
natural definition of a lift $F_{Y/X}$ in the case when $g$ is a
general proper local embedding, and $Y$ is reducible.

For any  morphism of Deligne--Mumford stacks $g: Y\to X$, the
fibered product $Y\times_XY$ represents the functor of isomorphisms
in $X$ of objects coming from $Y$: its objects over a scheme $S$ are
tuples $(\xi_1, \xi_2, f)$, where $\xi_1, \xi_2$ are objects in
$Y(S)$, and $f$ is an isomorphism between $g(\xi_1)$ and $g(\xi_2)$.
Let $\Delta: Y \to Y\times_XY$ denote the diagonal morphism and let
$Y^2_X$ denote the complement of its image in $Y\times_XY$. If $g$
is finite and unramified, then so are the projections $Y^2_X \to Y$,
and their image is the locus of points where $g$ is not \'etale on
its image. We can reiterate this construction with $(Y^2_X)^2_{Y}
\to Y^2_X$. Here $(Y^2_X)^2_{Y}$ is isomorphic to the complement
$Y^3_X$ of all diagonals in $Y\times_XY\times_XY$, and as such it
admits three different projections to $Y^2_X$. By successively
reiterating this construction until we reach $Y^{n+1}_X= \emptyset,$
we obtain a canonical network $\cN^n(Y/X)$ of local embeddings, the
last one of which is \'etale on its image. This network commutes
with base change, and it encapsulates the local \'etale structure of
the morphism $g:Y\to X$ in a way which is simultaneously
comprehensive and non-redundant.

In a sequence of steps, the network  $\cN^n(Y/X)$ can be replaced by
another network $\cN^0(Y/X)$ where all morphisms are closed
embeddings, and the objects admit \'etale, universally closed
surjections to the objects of $\cN^n(Y/X)$. The target of the new
network is $F_{Y/X}$. Moreover, the other objects of $\cN^0(Y/X)$
are also canonical lifts for the local embeddings contained in
$\cN^n(Y/X)$. Thus, for practical purposes the morphism $g$ can be
replaced by a set of copies of the closed embedding $F_{Y^2_X/Y}\to
F_{Y/X}$. The functorial presentation and properties of $F_{Y/X}$
are listed in Theorem \ref{main} and the Definition preceding it.

In \cite{rydh}, David Rydh constructed a different canonical lift
$E_{Y/X}$ for any unramified morphism $g:Y\to X$: he showed that $g$
has a universal factorization $Y \to E_{Y/X}\to X$, where the first
 morphism is a closed embedding $i$ and the second is \'etale;
 moreover, $E_{Y/X}$ comes with an open immersion $j:X \to E_{Y/X}$ such that
$i(Y)$ is the complement of $j(X)$ in $E_{Y/X}$. His construction
works in a more general context than ours, and indeed it was meant to
address the lack of an intrinsic presentation for our \'etale lift
in \cite{noi1}. However $E_{Y/X}$ differs from $F_{Y/X}$ in its range of applicability.  We would
like to thank David Rydh for his useful observations.

 The authors were supported by a Science
Foundation Ireland grant.

\section{The universal lift of a local embedding}

The stacks in this article are assumed to be algebraic in the sense
of Deligne--Mumford,  Noetherian, and all morphisms considered
between them are of finite type.


\subsection{The lift of a local embedding \'etale on its image.}

\begin{definition}
Following \cite{vistoli}, we will call local embedding any
representable unramified morphism of finite type of stacks. A
regular local embedding is a local embedding which is also locally a
complete intersection.
\end{definition}



\begin{proposition}\label{lift for morph etale on its image}
Let $g: Y \to X$ be a proper morphism of stacks  \'etale on its image.
There exists an \'etale morphism $e_g: F_{Y/X} \to X$ together with
an isomorphism $\phi:
g(Y)\times_XF_{Y/X} \to Y$,  such that
\begin{itemize}
\item[(0) i)]  the triangles in the following diagram are
commutative
\bea \diagram  {Y\times_X F_{Y/X}} \dto_{g\times \mbox{id}_{F_{Y/X}}} \rto^{} & {Y} \dto^{g} \\
  {g(Y)\times_XF_{Y/X}} \rto \urto^{\phi} & {X},  \enddiagram \eea
  where the upper horizontal arrow is the projection on $Y$ and the lower horizontal arrow
  is the restriction of $e_g$ to
  $g(Y)\times_XF_{Y/X}$;
 \item[(0) ii)]  Let  $p_2: g(Y)\times_XF_{Y/X}\to F_{Y/X}$ be the projection on the second factor and consider the closed embedding $i:= p_2\circ \phi^{-1}$. The
 restriction of $e_g$ induces an isomorphism $F_{Y/X}\setminus i(Y) \to X\setminus g(Y)$.
  \end{itemize}
 The following properties also hold:
\begin{itemize}
\item[(1)] For any stack $Z$, there is an equivalence of categories between $\Hom (Z, F_{Y/X})$ and  the category of
morphisms $Z \to X$ endowed with a section
$$s: g(Y) \times_X Z \to Y \times_X Z $$
for the \'etale map $Y \times_X Z  \to g(Y) \times_X Z$.
\item[(2)] The triple $(F_{Y/X}, e_g, \phi)$, with $e_g$ \'etale and satisfying (0)i) and (0)ii) is uniquely defined up to unique 2-isomorphism.
\item[(3)] The morphism $e_g: F_{Y/X} \to X$ is universally closed.
\item[(4)] If $g:Y\to X$ is a closed embedding, then $F_{Y/X}\cong X$.
\item[(5)] If $g:Y\to X$ is \'etale and proper, and $X$ is connected then $F_{Y/X}\cong Y$.
\item[(6)] For any morphism of stacks $u: X' \to X$ and $Y':= Y\times_{X}X'$,
there exists a morphism $F_u:F_{Y'/X'} \to F_{Y/X}$ making the squares in the following diagram Cartesian:
\bea     \diagram    Y' \dto \rto & {F_{Y'/X'}}  \rto \dto^{F_u} & {X'} \dto^{u}  \\
Y \rto & {F_{Y/X}} \rto & {X}. \enddiagram  \eea
\item[(7)] If $h: Z \to Y$ is proper and \'etale on its image, and $g:Y\to X$ is a closed embedding, then $F_{Z/Y}\cong Y\times_XF_{Z/X}$. In particular, there exists a natural \'etale morphism $g_*:F_{Z/Y} \to F_{Z/X}$.
\item[(8)] For any morphism  $h: Z \to Y$ proper and \'etale on its image, the composition  morphism $g\circ i_h: F_{Z/Y} \to X$, universally closed and  \'etale on its image, comes with an \'etale map $F_{F_{Z/Y}/X} \to X$ satisfying properties (0) and (1). Moreover,
    \bea F_{F_{Z/Y}/X} \cong F_{Z/F_{Y/X}}.\eea
    In particular, if the morphism $h: Z \to Y$ is \'etale, then there exists a natural morphism $h_*: F_{Z/X} \to F_{Y/X}$.
\item[(9)] If $h: Z \to Y$ and  $g:Y\to X$  are proper and \'etale on their images, and if  $g(h(Z))\times_XY\cong h(Z)$ over $Y$,
then there exists a morphism $g_*:F_{Z/Y} \to F_{Z/X}$ such that $e_{g\circ h}\circ \bar{g}=g\circ e_h$. In this case
 $  F_{F_{Z/Y}/F_{Z/X}} \cong F_{F_{Z/Y}/X} \cong F_{Z/F_{Y/X}}.$
 \item[(10)] Given any proper local embeddings $g: Y\to X$ and $f: T\to X$, and $Z:= Y\times_X T$, there are natural isomorphisms
 \bea F_{Y/X}\times_XF_{T/X}  \cong F_{F_{Z/T}/F_{Y/X}} \cong  F_{F_{Z/Y}/F_{T/X}} \cong F_{F_{Z/Y}\bigcup_{Z}F_{Z/T}/F_{Z/X}},\eea
 where $F_{Z/Y}\bigcup_{Z}F_{Z/T}$ denotes the stack obtained by gluing the stacks $F_{Z/Y}$ and $F_{Z/T}$ along $Z$.
\end{itemize}
\end{proposition}

\begin{proof}
An explicit \'etale groupoid presentation for a functor $F_{Y/X}$
which satisfies properties (0) and (1) was found in \cite{noi1},
section 1.1. We briefly recall it here. One chooses an \'etale cover
by a scheme $p:U \to X$ such that  $Y\times_XU\cong V=V_1\bigsqcup
V_2$ where $V_1=g(Y)\times_X U$.
 Let $$S_{ij}:=\Im (\phi_{ij}: V_i\times_Y V_j \to U\times_X U ),$$
for the map $\phi_{ij}$ given as a composition
 $$V_i\times_Y V_j \hookrightarrow V\times_Y V = V\times_Y(
 Y\times_X U) \cong V \times_X U \to U\times_X U.$$
  A groupoid presentation of  $F_{Y/X}$ is given by $\left[ R' \rightrightarrows U\right]$
$$R':= (U\times_XU) \setminus (S_{12}\cup S_{21}\cup (S_{22}\setminus S_{11})) \cup \Im e.$$

To prove property (2), we note that for any triple $(F', e', \phi')$ satisfying the properties (0), there is a canonical section $g(Y)\times_XF'\to Y\times_XF'$ which, together with the map $e'$, determine a unique morphism $u: F'\to F_{Y/X}$ such that  $e\circ u=e'$, due to condition (1). Both $e$ and $e'$ are \'etale, and so $u$ must be \'etale as well. On the other hand, $e$ and $e'$ induce isomorphisms $g(Y)\times_XF_{Y/X}\cong Y\cong g(Y)\times_XF'$ and $F_{Y/X}\setminus i(Y)\cong X\setminus g(Y) \cong F'\setminus i'(Y)$. Thus $u$ is both \'etale and bijective, and so an isomorphism.

Property (3) follows from the valuation criterium in conjunction
with property (1). Consider a complete discrete valuation ring $R$
with field of fractions $K$, a commutative diagram
\bea \begin{CD} \Spec (K) @>{u} >> F_{Y/X}\\
                 @VV{\rho} V  @VV{p} V \\
                 \Spec(R) @>{v}>> X.   \end{CD} \eea
The closed embedding $g(Y)\times_X\Spec (K)\to \Spec (K)$ is either
the empty embedding or an isomorphism. If empty, then $u$ factors
through $\Spec (K)\to  F_{Y/X}\setminus Y \cong X\setminus g(Y)$ and
so $v$ also naturally yields $\Spec (R)\to X\setminus g(Y) \cong
F_{Y/X}\setminus Y$. If an isomorphism,  then the map $v$ induces a
natural morphism $ \Spec (K)\cong g(Y)\times_X\Spec (K) \to Y$ whose
composition with $g$ is $u\rho$. As $g$ is proper, there is a lift
$\Spec (R) \to Y$, which yields a section $g(Y)\times_X \Spec (R)
\to Y\times_X \Spec (R)$.  This, together with the map $u\to \Spec
(R) \to X$ give the data for a unique morphism $\Spec (R) \to
F_{Y/X}$ as required.

 Properties (4) and (5) are direct consequences of (2).

 Property (6) was proven in \cite{noi1}, Corollary 1.8. Alternatively, it follow  immediately from (2). Indeed, consider a morphism of stacks $f: X' \to X$ and let $Y':=
Y\times_{X}X'$, with the morphism $g':Y'\to X'$ induced by $g$. Then
the \'etale morphism  $X'\times_XF_{Y/X} \to X'$ induced by $e_g$,
together with the composition
 \bea & g'(Y')\times_{X'}(X'\times_XF_{Y/X}) \cong (g(Y)\times_XX')\times_{X'}(X'\times_XF_{Y/X}) \cong & \\ & \cong  g(Y)\times_{X}X'\times_XF_{Y/X}\cong X'\times_XY\cong Y',& \eea
satisfy properties (0) for the morphism $g':Y'\to X'$ and so $X'\times_XF_{Y/X} \cong F_{Y'/X'}$. Thus
\bea & Y'\cong g'(Y')\times_{X'}F_{Y'/X'}\cong (g(Y)\times_XX')\times_{X'}F_{Y'/X'}\cong & \\ &\cong g(Y)\times_XF_{Y'/X'}\cong (g(Y)\times_XF_{Y/X})\times_{F_{Y/X}}F_{Y'/X'}\cong Y\times_{F_{Y/X}}F_{Y'/X'}.&\eea

 To prove (7), we will construct a canonical morphism $F_{Z/Y} \to Y\times_XF_{Z/X}$, together with its inverse. To construct $F_{Z/Y} \to F_{Z/X}$, we first consider the composition
 $g\circ e_h: F_{Z/Y}\to Y\to X$. The canonical isomorphisms
 \bea & g(h(Z))\times_XF_{Z/Y}\cong (h(Z)\times_XY) \times_YF_{Z/Y} \cong h(Z)\times_YF_{Z/Y} \cong Z,& \mbox{ and } \\
 & Z\times_XF_{Z/Y}\cong (Z\times_XY) \times_YF_{Z/Y} \cong Z\times_YF_{Z/Y},&\eea
together with the embedding $Z \to Z\times_YF_{Z/Y}$ give a section $g(h(Z))\times_XF_{Z/Y} \to Z\times_YF_{Z/Y}$, and thus, according to (1), a map  $F_{Z/Y} \to F_{Z/X}$. This, together with the \'etale map $F_{Z/Y} \to Y$ generate the desired morphism $_{Z/Y} \to Y\times_XF_{Z/X}$. Its inverse is also constructed via property (1) as follows: We consider the projection $ Y\times_XF_{Z/X} \to Y$ together with the canonical section \bea h(Z)\times_Y( Y\times_XF_{Z/X})\cong h(Z)\times_XF_{Z/X} \to Z\times_XF_{Z/X} \cong  Z\times_Y( Y\times_XF_{Z/X}).\eea

Proof of (8): Consider now a morphism  $h: Z \to Y$ proper and \'etale on its
image. We will show that $e_g\circ e_{i\circ h}: F_{Z/F_{Y/X}} \to X$ is an 'etale lift for the composition $g\circ i_h: F_{Z/Y} \to X$. Note that  $g\circ i_h$  is universally closed and \'etale on its image  $\Im g\circ i_h = g(Y)$, though not necessarily separated.
Let $i:Y \to F_{Y/X}$ be the natural embedding induced by $g$. Then $i\circ h: Z \to  F_{Y/X}$ is the composition of a
proper morphism \'etale on its image and a closed embedding. Due to (7) applied to this composition,  there are Cartesian diagrams
\bea    \diagram    {F_{Z/Y}}  \dto \rto^{e_h} & {Y}  \rto \dto & {g(Y)} \dto  \\
F_{Z/F_{Y/X}} \rto & {F_{Y/X}} \rto^{e_g} & {X}, \enddiagram  \eea
whose composition implies that $F_{Z/Y} \cong F_{Z/F_{Y/X}}\times_X g(Y)$ canonically and that $e_g\circ e_{i\circ h}$ induces  $ F_{Z/F_{Y/X}} \setminus F_{Z/Y} \cong X\setminus g(Y)$.
  We also check that $e_g\circ e_{i\circ h}$ satisfies property (1), namely that any map $T \to F_{Z/F_{Y/X}}$ is uniquely determined by a pair of maps $f:T\to X$ and a section $s: g(Y)\times_XT\to F_{Z/Y}\times_XT$. Indeed, such a pair, together with the composition $g(Y)\times_XT\to F_{Z/Y}\times_XT \to Y\times_XT$, determine in a first instance a map $T \to F_{Y/X}$, whose composition with $e_g$ yields $f$. The restriction of $s$ also yields a section $g(h(Z))\times_XT\to Z\times_XT$ and so a sequence of morphisms over $F_{Y/X}$:
\bea i(h(Z))\times_{F_{Y/X}}T \to g(h(Z))\times_XT \to Z,  \eea
and so a section $i(h(Z))\times_{F_{Y/X}}T \to Z\times_{F_{Y/X}}T$. This determines a map $T\to F_{Z/F_{Y/X}}$ whose composition with $e_g\circ e_{i\circ h}$ yields $f$.

Proof of (9): The isomorphisms
\bea  g(h(Z))\times_XF_{Z/Y} \cong (g(h(Z))\times_X Y)\times_YF_{Z/Y}\cong h(Z)\times_YF_{Z/Y}\cong Z,\eea
together with the composition $g\circ e_h$, give a morphism $F_{Z/Y}\to F_{Z/X}$. Clearly $F_{F_{Z/Y}/F_{Z/X}}$ satisfies properties (0) as an \'etale lift of $g\circ e_h$, hence the isomorphism
$F_{F_{Z/Y}/F_{Z/X}}\cong F_{F_{Z/Y}/X}$.

Proof of (10): The first two isomorphisms are direct consequences of property (6). Indeed,
\bea F_{F_{Y\times_XT/T}/F_{Y/X}} \cong F_{F_{Y/X}\times_XT/F_{Y/X}} \cong F_{Y/X}\times_X F_{T/X}, \eea
and similarly for $F_{F_{Y\times_XT/Y}/F_{T/X}}$. To prove the last isomorphism, we will first need to pinpoint the existence of a natural local embedding
$F_{Z/Y}\bigcup_{Z}F_{Z/T} \to F_{Z/X}$. Indeed, via property (9), there exist compositions
\bea F_{Z/T} \hookrightarrow F_{F_{Z/T}/X} \cong F_{Z/F_{T/X}} \to F_{Z/X} \mbox{ and }  F_{Z/Y} \hookrightarrow F_{F_{Z/Y}/X} \cong F_{Z/F_{Y/X}} \to F_{Z/X}.\eea
Indeed, the hypotheses necessary for property (9) hold because $Z=Y\times_XY$.
Moreover, the compositions above, together with the embeddings of $Z$ into $F_{Z/T}$ and $F_{Z/Y}$, respectively, form a commutative diagram, which insure the existence of the morphism
$F_{Z/Y}\bigcup_{Z}F_{Z/T} \to F_{Z/X}$ (conform \cite{abramovich}, Appendix 1). Moreover, by construction this morphism is proper and a local embedding.

 In a similar way we can check the existence of a closed embedding \bea j:F_{Z/Y}\bigcup_{Z}F_{Z/T} \to F_{Y/X}\times_XF_{T/X}.\eea Indeed,
the isomorphisms
  $F_{Y/X}\times_XF_{T/X}  \cong F_{F_{Z/T}/F_{Y/X}} \cong  F_{F_{Z/Y}/F_{T/X}}$
 implicitly state the existence of closed embeddings of $F_{Z/Y}$ and $F_{Z/T}$ into $F_{Y/X}\times_XF_{T/X}$, which commute with the embeddings of $Z$
into $F_{Z/T}$ and $F_{Z/Y}$ respectively, and thus define the closed embedding $j$.

  Furthermore, property (9) implies the existence of natural morphisms from $F_{Z/Y}$ and $F_{Z/T}$ to $F_{Z/X}$, which induce a natural morphism $F_{Z/Y}\bigcup_{Z}F_{Z/T} \to F_{Z/X}$.

  The existence of a natural morphism $e: F_{Y/X}\times_XF_{T/X} \to F_{Z/X}$ follows from the universal property (1) of $F_{Z/X}$. Indeed, via the canonical \'etale morphism  $F_{Y/X}\times_XF_{T/X} \to X$, there are natural morphisms
  \bea  \Im (Z \to X)\times_XF_{Y/X}\times_XF_{T/X}  \hookrightarrow \Im f\times_XF_{T/X}\times_XF_{Y/X} \to T\times_XF_{Y/X}\times_XF_{T/X},\eea
  and similarly
   \bea  \Im (Z \to X)\times_XF_{T/X}\times_XF_{T/X}  \hookrightarrow \Im g\times_XF_{Y/X}\times_XF_{T/X} \to Y\times_XF_{Y/X}\times_XF_{T/X},\eea
  forming a commutative diagram with the projections to $F_{Y/X}\times_XF_{T/X}$, and thus inducing a section
  \bea  \Im (Z \to X)\times_XF_{Y/X}\times_XF_{T/X}  \to Z \times_XF_{Y/X}\times_XF_{T/X}.  \eea
  This proves the existence of the natural morphism $e: F_{Y/X}\times_XF_{T/X} \to F_{Z/X}$, which is \'etale because the natural maps from both its target and source to $X$ are \'etale. We have thus obtained a diagram
  \bea \diagram     && F_{Y/X}\times_XF_{T/X} \dto^{e} \\
  F_{Z/Y}\bigcup_{Z}F_{Z/T} \urrto^{j} \rrto && F_{Z/X} , \enddiagram  \eea
  which is commutative due to the natural choices of the morphisms and property (1).

  By (0) and (2), it remains to show that the complement of $ F_{Z/Y}\bigcup_{Z}F_{Z/T} $ in $F_{Y/X}\times_XF_{T/X}$ is naturally isomorphic to the complement of
  $\Im ( F_{Z/Y}\bigcup_{Z}F_{Z/T} \to F_{Z/X})$ in $F_{Z/X}$, and that there is a natural isomorphism
  \bea  \Im (F_{Z/Y}\bigcup_{Z}F_{Z/T} \to F_{Z/X})\times_{F_{Z/X}}(F_{Y/X}\times_XF_{T/X}) \cong F_{Z/Y}\bigcup_{Z}F_{Z/T}.\eea
 These properties follow canonically from the definitions of the objects and morphisms involved.

\end{proof}

\begin{lemma} \label{split into etale and generically deg. 1}
Let $g: Y \to X$ be a proper local embedding of Noetherian stacks, with $Y$ integral. Then there exists a stack $D_{Y/X}$ together with an \'etale epimorphism $e: Y\to D_{Y/X}$ and a proper local embedding $g_1: D_{Y/X}\to X$ of generic degree 1, such that $g=g_1\circ e$. Moreover, $D_{Y/X}$ is unique up to an isomorphism.
\end{lemma}

\begin{proof}
A factorization of the morphism  $g: Y \to X$ into an \'etale epimorphism  $e:Y\to D_{Y/X}$ and a proper local embedding $g_1: D_{Y/X}\to X$ of generic degree one
 was constructed in \cite{noi1}, Lemma 1.10. It remains to prove uniqueness up to an isomorphism. For this, we first recall the \'etale local structure of $D_{Y/X}$: There exists an \'etale cover of $X$  by a scheme $U$ such that $Y\times_XU = \bigsqcup_{i,a}V_i^a$, and for each $i,a$, the morphism $g_U: Y\times_XU \to U$ restricts to a closed embedding $V_i^a \hookrightarrow U$, with image $W_i$, such that $W_i\not= W_j$ if $i\not= j$. Let $W=\bigcup_iW_i$. There exists a canonical groupoid structure $ \left[s_e, t_e: R_{e}:= \bigsqcup_iW_i \times_XU \rightrightarrows  \bigsqcup_iW_i  \right]$, and $D_{Y/X}$ is defined as its associated stack. The morphisms $e: Y\to D_{Y/X}$ and $g_1: D_{Y/X}\to X$, respectively, are determined by the  canonical choice of  maps $e_U: \bigsqcup_{i,a}V_i^a \to \bigcup_iW_i$ and $g_{1 U}: \bigcup_iW_i \to U$, together with $e_R: \bigsqcup_{i,a, j, b}V_i^a\times_Y V_j^b \to \bigsqcup_iW_i \times_XU$ and $g_{1 R}:  \bigsqcup_iW_i \times_XU \to U\times_X U$ at the level of relations.

 Assume that $e':Y\to Y'$ is another \'etale epimorphism, and that $f':Y'\to X$ is a proper local embedding of generic degree one such that $f'\circ e' = g$. Let $\bigcup_iV'_i := Y'\times_X U$, with the induced morphism $f'_U: \bigcup_iV'_i \to U$, such that each $V'_i$ is the preimage of $W_i$. As the induced morphism $e'_U: \bigsqcup_{i,a}V_i^a \to \bigcup_iV'_i$ is \'etale and surjective and sends each $V_i^a$ to $V'_i$, the components $V'_i$ must be pairwise disjoint. We will construct an isomorphism of groupoids
  \bea  \phi:  \left[s_e, t_e: R_{e}:= \bigsqcup_iW_i \times_XU \rightrightarrows  \bigsqcup_iW_i  \right]    \rightarrow \left[s', t': R':= \bigsqcup_iV'_i \times_{Y'} \bigsqcup_iV'_i \rightrightarrows  \bigsqcup_iV'_i  \right].  \eea
First consider any section  $\sigma: \bigsqcup_iW_i \to \bigsqcup_{i,a}V_i^a$ of $e_U$ and define $\phi_U:= e'_U\circ  \sigma$. As section of the \'etale morphism $e_U$,  the map $\sigma$ must be \'etale itself. In fact, it consists of a choice of an index $a$ for each $i$, and an isomorphism $W_i\to V_i^a$. The map $e'_U$ is \'etale and surjective, and it maps each $V_I^a$ onto $V'_i$. Indeed, $\deg g_{1 U}\circ e_U = \deg f'_U\circ e'_U$ while $\deg g_{1 U}$, $f'_U$ are both of generic degree one, so the image of each $V_I^a$ under $e'_U$ must be a dense open subset of $V'_i$. On the other hand, $e'_U$ is also proper, so $e'_U(V_I^a)=V'_i$.

It follows that $\phi_U= e'_U\circ  \sigma$ is \'etale and surjective as well. Moreover, $f'_U\circ \phi_U =g_{1 U}$, so the degree of $\phi_U$ must be one.
Thus  $\phi_U$ is an isomorphism.

Let $R:=U\times_X U$, and consider the first projection $s:R\to U$. As $R' \cong \bigsqcup_iV'_i\times_U R$, we can construct $\phi_R: R_e \to R'$ as the morphism uniquely defined by the conditions
\bea  f'_R\circ \phi_R = g_{1 R} \mbox{ and } s'\circ \phi_R = \phi_U \circ s_e. \eea
Similarly, a morphism $\psi_R: R'\to R_e$ can be defined by the conditions
\bea  g_{1 R}\circ \psi_R = f'_R \mbox{ and } s_e\circ \psi_R = \phi^{-1}_U \circ s'. \eea
We note that $\psi_R\circ\phi_R =\mbox{ id}_{R_e}$, as  $g_{1 R}\circ \psi_R\circ\phi_R = g_{1 R}$ and $s_e\circ \psi_R\circ\phi_R = s_e$, and $R_e\cong \bigsqcup_iW_i \times_U R$. Similarly, $\phi_R\circ\psi_R =\mbox{ id}_{R'}$.

It remains to prove that the pair $(\phi_U, \phi_R)$ is a morphism of groupoids. This is a slightly long, but direct check. Here we will prove the equality:
  \bean \label{t} \phi_U\circ t_e = t'\circ \phi_R.\eean
Let $\begin{array}{ll} i_e: R_e \to R_e, & i':R'\to R'\end{array}$ and $i:R\to R$ denote the inverting maps of the groupoids $\begin{array}{ll} [R_e \rightrightarrows \bigsqcup_iW_i], & [R'\rightrightarrows \bigsqcup_iV'_i] \end{array} $ and $ [R\rightrightarrows U ]$ respectively, so that $i_e\circ s_e=t_e,$ $ i'\circ s'=t' $ and $i\circ s=t$. Composition with $f'_U$ of the two terms in the equation (\ref{t}) yields:
\bea & f'_U\circ \phi_U\circ t_e =g_{1 U}\circ t_e, \mbox{ and }\\
& f'_U\circ  t'\circ \phi_R=  f'_U\circ   i'\circ s' \circ \phi_R=  f'_U\circ i'\circ \phi_U \circ s_e = & \\
& = i\circ f'_U\circ \phi_U \circ s_e =  i\circ g_{1 U}\circ s_e =g_{1 U}\circ i_e\circ s_e= g_{1 U}\circ t_e.   \eea
As $f'_U$ is generically injective, the closed subset $\begin{array}{ll} \{ x\in R_e; &  \phi_U\circ t_e(x) = t'\circ \phi_R(x)  \} \end{array}$ contains an open dense subset of $R_e$, so it must be the entire $R_e$.

All other compatibility relations follow directly by the same method as above.

\end{proof}

\begin{remark}
We note that if $g: Y \to X$ was not separable, the uniqueness of a possible split $Y\to D_{Y/X} \to X$ would not be guaranteed. For example, if $p\not=q$ are natural numbers and $Y$ is obtained by gluing $pq$ copies of the space $X$ along the complement of a point, then two possible choices for $D_{Y/X}$ would be obtained by gluing $p$, respectively $q$ copies of the space $X$ along the complement of that same point.
\end{remark}

\begin{proposition}\label{properties of split}
Let $g: Y \to X$ be a proper local embedding of Noetherian stacks, with $Y$ integral.
The stack $D_{Y/X}$ constructed above satisfies the following properties:
\begin{itemize}
\item[(1)] For any morphism of stacks $u: X' \to X$ and $Y':= Y\times_{X}X'$,
there exists a morphism $D_u: D_{Y'/X'} \to D_{Y/X}$ making the squares in the following diagram Cartesian:
\bea     \diagram    Y' \dto \rto^{e'} & {D_{Y'/X'}}  \rto^{f'} \dto^{D_u} & {X'} \dto^{u}  \\
Y \rto^{e} & {D_{Y/X}} \rto^{f} & {X}. \enddiagram  \eea
\item[(2)] If $h: Z \to Y$ is another proper local embedding of integral Noetherian stacks, then there exists a natural isomorphism
\bea D_{D_{Z/Y}/D_{Y/X}} \cong D_{Z/X},\eea
where $D_{D_{Z/Y}/D_{Y/X}} $ is the stack associated to the composition $e\circ h_1$ of  the local embedding of generic degree one $h_1: D_{Z/Y}\to Y$  and the \'etale map $e:Y\to D_{Y/X}$.
\item[(3)] There is an equivalence of categories between the category of commutative diagrams
\bea \diagram  Y \rto^{g}  & X \\
               T \uto^{v} \rto^{f} & Z \uto^{u},\enddiagram   \eea
 with $f$ \'etale, and that of pairs in $\Hom(Z, D_{Y/X})\times \Hom(T, Z\times_{D_{Y/X}}Y)$ such that the induced morphism $T \to Z$ is \'etale.

 The morphisms in the first category are given by Cartesian diagrams \bea \diagram    T' \dto^{t} \rto & Z'\dto^{z} \\ T \rto & Z \enddiagram \eea
such that $t$ and $z$ commute with the given morphisms to $Y$ and $X$, respectively.
\end{itemize}
\end{proposition}

\begin{proof}
Properties (1) and (2) are direct consequences of the definition of $D_{Y/X}$ and Lemma \ref{split into etale and generically deg. 1}. The proof of property (3) is based on arguments also employed in the proof of same Lemma. Indeed, given an \'etale cover of $X$  by a scheme $U$ such that $Y\times_XU = \bigsqcup_{i,a}V_i^a$, and for each $i,a$, the morphism $g_U: Y\times_XU \to U$ restricts to a closed embedding $V_i^a \hookrightarrow U$, with image $W_i$, such that $W_i\not= W_j$ if $i\not= j$. Then $D_{Y/X}\times_XU\cong \bigsqcup_iW_i$.  Also, $T\times_Y(\bigsqcup_{i,a}V_i^a)$, and as the morphism $f:T\to Z$ is \'etale, then $Z\times_XU\cong \bigsqcup V'_i$ for some $V'_i$-s such that for each $i$, the pullback of $f$ restricts to maps $\bigsqcup_{a}V_i^a\to \bigsqcup V'_i$, and the pullback of $u$ restricts to $ V'_i \to W_i$. In particular, this induces a map $ \bigsqcup V'_i \to \bigsqcup_iW_i$. A morphism of groupoids
\bea [(\bigsqcup V'_i)\times_Z(\bigsqcup V'_i)\rightrightarrows \bigsqcup V'_i]  \to [ (\bigsqcup_iW_i)\times_{D_{Y/X}}(\bigsqcup_iW_i) \rightrightarrows \bigsqcup_iW_i]\eea
can then be constructed by the exact same method as in the proof of the previous Lemma.
\end{proof}

\subsection{ }

In the next paragraphs we will work with simple categories whose objects are Noetherian stacks, and such that there exists at most one morphism between each pair of objects. We will discuss some additional properties below.

\begin{definition}
By extending the terminology of \cite{lunts}, we can define a poset of stacks as follows. We regard any poset $\cP$ as a category, such that for any elements $I, J \in \cP$, the set of morphisms $\mbox{ Morph}(I,J)$ consists of a unique element if $I \leq J$, and is empty otherwise. Then a poset of stacks is a contravariant functor from $\cP$ to the category of sets. Any such poset of stacks $\cN=\{\phi_J^I: Y_J \to Y_I\}_{I\subseteq J\in \cP}$ where $\cP$ is the power set of a finite set $\Lambda$, and the partial order is given by inclusion, will be called simply a network. In particular, a network will include a unique target $Y_{\emptyset}$, (and a source $Y_{\Lambda}$, possibly empty).

\end{definition}

\begin{definition}

Let $\cN=\{\phi_J^I: Y_J \to Y_I\}_{I\subseteq J\in \cP}$ be a network of morphisms with target $X=Y_{\emptyset}$, and let $\cN'=\{\phi'^I_J: Y'_J \to Y'_I\}_{I,J\in \cP'}$ be another network with target $X'=Y_{\emptyset}$, where  $\cP'\subseteq \cP$. A morphism of networks $F: \cN' \to \cN$ is a fully faithful functor from the category $\cN'$ to the category $\cN$,
given by a set of morphisms $\{ f_I:Y'_I \to Y_I\}_{I\in \cP}$, such that $f_I\circ  \phi'^I_J = \phi_J^I\circ f_J$. In particular, $F$ includes a morphism between targets $f:X'\to X$.

We say that $\cN'\cong \cN\times_XX'$ if each of the commutative diagrams corresponding to the equalities $f_I\circ  \phi'^I_J = \phi_J^I\circ f_J$ is Cartesian.

\end{definition}

Given a network of closed embeddings, there is a natural way to glue any subset of objects $\{ Y_I \}_{I \in \cQ}$ into a stack $S_\cQ$ as follows:

\begin{lemma} \label{gluing in a network}
Let $\cN=\{\phi_J^I: Y_J \to Y_I\}_{I\subseteq J\in \cP}$ be a network of closed embeddings, with target $X$. Consider  $\cQ\subseteq \cP$.

 a) There exists a stack $S_{\cQ}$, and commutative diagrams
\bea \diagram
Y_{I\cup J} \dto_{\phi_{I\cup J}^J} \rto^{\phi_{I\cup J}^I} & Y_{J} \dto \\
Y_{I} \rto & S_{\cQ}
\enddiagram \eea
for all $I, J \in {\cQ}$, such that for any stack $T$, the
 natural functor
 \bea \Hom(S_{\cQ}, T) \to \times_{\{ \Hom(Y_{I\cup J}, T)\}_{I,J \in {\cQ}}}\{ \Hom(Y_J, T)\}_{J\in {\cQ}} \eea
 is an equivalence of categories.

 In particular, there exists a natural morphism $S_{\cQ} \to Y_{(\bigcap_{I\in \cQ}I)}$ compatible with the morphisms $\phi_{I}^{(\bigcap_{I\in \cQ}I)}$, for all $I\in \cQ$.

 Such a stack is unique up to unique isomorphism.

 b) For any morphism $X' \to X$, consider the network  $\cN':= \cN \times_XX'$, with objects $Y'_I= Y_I\times_XX'$. Consider the stack $S'_{\cQ}$, obtained by gluing the objects $\{Y'_I\}_{I\in \cQ}$ in the network $\cN'$. Then for all $I\in \cQ$, the squares in the following diagram are Cartesian:
 \bea \diagram  Y'_I \rto \dto & S'_{\cQ}  \rto \dto & Y'_{(\bigcap_{I\in \cQ}I)} \dto \\
   Y_I \rto  & S_{\cQ}  \rto  & Y_{(\bigcap_{I\in \cQ}I)}. \enddiagram \eea

\end{lemma}

\begin{proof}
a) We will proceed by induction on the cardinality of ${\cQ}$. If $|{\cQ}|=1$, then $S_{\cQ}=Y_I$ for $I \in {\cQ}$. Assume now that $S_{\cQ}$ exists for any ${\cQ}$ of a given cardinality. Fix such ${\cQ}$ and let $J\not\in {\cQ}$. Note that if $J \supseteq I$ for some $I \in {\cQ}$, then $S_{{\cQ}\cup\{J\}}=S_{\cQ}$. If this is not the case, let ${\cQ}^J:=\begin{array}{ll} \{ I \bigcup J;   & I\in {\cQ} \} \end{array}$. Then by induction, $S_{{\cQ}^J}$ exists and, moreover, there is a unique closed embedding $S_{{\cQ}^J} \to S_{\cQ}$ determined by the compositions $Y_{J\cup K} \to Y_{K} \to S_{\cQ}$ for all $K\in {\cQ}$. Gluing $Y_J$ and $S_{\cQ}$ along
$S_{{\cQ}^J}$ yields a stack satisfying the required properties (conform \cite{abramovich}, Proposition A.1.1).

Part b) follows by standard category theoretical arguments. Indeed, since $Y'_I\cong Y_I\times_{Y_{(\bigcap_{I\in \cQ}I)}}Y'_{(\bigcap_{I\in \cQ}I)}$, it is enough to show that the right side of the diagram is Cartesian. Given two morphisms $T\to S_{\cQ}$ and $T\to Y'_{(\bigcap_{I\in \cQ}I)}$ commuting to the respective morphisms to $Y_{(\bigcap_{I\in \cQ}I)}$, we can think of $T$ as being obtained by gluing the objects $\{T\times_{S_{\cQ}}Y_I\}_{I\in \cQ}$ within the network whose objects are $\{T\times_{S_{\cQ}}Y_J\}_{J \supseteq I \mbox{ for some } I \in \cQ}$, and the target $T$. Each such object $T\times_{S_{\cQ}}Y_J$ admits two natural morphisms, to $Y_I$ and $Y'_{(\bigcap_{I\in \cQ}I)}$, respectively, commuting to the respective morphisms to $Y_{(\bigcap_{I\in \cQ}I)}$, and thus admits natural morphisms $T\times_{S_{\cQ}}Y_J\to Y'_I \to S'_{\cQ}$, for $J\supseteq I \in \cQ$. By part a), there exists a unique natural morphism $T \to S'_Q$ compatible with $T\to S_{\cQ}$ and $T\to Y'_{(\bigcap_{I\in \cQ}I)}$.

\end{proof}

Consider a proper local embedding of Noetherian stacks
$g: Y \to X$, with $Y$ integral. Starting from the flat stratification of $g$,  in \cite{noi1}, we constructed a
network of local embeddings associated to $g$, and an \'etale lift $F_{Y/X}\to X$ which reflected the local \'etale
structure of the morphism $g$. However, this lift was not canonical, as it depended of the choice of \'etale cover by
a nice scheme $U$ of $X$. We recall here the properties of $U$ which were essential for the construction of $F_{Y/X}$.

Let $ Y_n \hookrightarrow Y_{n-1}
\hookrightarrow ... \hookrightarrow  Y_1 \hookrightarrow Y_0=Y$ be a
filtration of $Y$ consisting of the closures $\ol{g^{-1}(S_i)}\subseteq Y$,  where $\{S_i\}_i$ is the flattening stratification for
the morphism $Y \to g(Y)$.

\begin{definition} \label{U} Let $g: Y \to X$ be a proper local embedding of Noetherian stacks.
An \'etale cover $U$ of $X$ is called suitable for the morphism $g$ if the following properties hold:
\begin{enumerate}
\item $g(Y)\times_XU = \bigcup_{l\in L}  W_l$, where $W_l$ are isomorphic, for all $l\in L$.
\item For all $k=0, ..., n$, and for some suitable choices of subsets $\cP_k \subset \cP:=\cP(L)$, we have $ \bigcup_{I \in \cP_k} W_I = g(Y_k)\times_XU $, where
 $W_I= \bigcap_{l\in I} W_l$, and $W_I \cong W_{I'}$ for all $I, I'\in \cP_k$.
\item For each $I$ as above, there exist sets
 $\{ V_I^{a} \}_{a\in A_I }$ mapping onto $Y_k$, with isomorphisms $ V_I^{a} \to W_I$ standing over $g_k: Y_k \to g(Y_k)$, and
 satisfying  $$Y_k\times_XU = \bigcup_{I\in \cP_k, a\in A_I }V_I^{a}.$$
  Here $A_I=\bigsqcup_{k\in I} A_k$ and $V_I^{a} \subseteq V_k^{a}$ if $k\in I$ and $a\in A_k$.
\end{enumerate}
\end{definition}

\begin{definition}\label{network} Let $X$ and $Y$ be Noetherian stacks, with $Y$ integral. Let $g: Y \to X$ be a proper local embedding of generic degree one.
  Let $U\to X$ be a suitable \'etale cover for $g$.
We associate to $g$ and $U$ a network of local embeddings
$\phi_J^I: Y_J \to Y_I$, one for each pair $I\subseteq J$, $I \in \cP_i$ and $J\in \cP_j$, as follows. For each $I\subseteq L$, for each distinct $i,j\in L$ and the uniquely associated indices $a\in A_i$, $b\in A_j$,
 we define \bea   \begin{array}{lll} R_{\emptyset}:= U\times_XU, &  R_i := V^a_i\times_YV^a_i & \mbox{ and }  R_I:=R_I = ( \prod_{i\in I} )_{R_{\emptyset} } R_i,\end{array} \eea
and $Y_I$ as the stack with groupoid presentation $ \left[  R_I \rightrightarrows   V_I^a \right].$ We consider by convention $V_{\emptyset}^a =U$, such that $Y_{\emptyset}=X$.
We note that \bea \label{R_I} R_I \cong W_I\times_X W_I \setminus \bigcup_{j\not= i \in I} S_{ij}^{ab}, \mbox{ where }  S_{ij}^{ab} := \Im ( V_i^a \times_Y V_j^b \to W_i \times_X W_j).  \eea
 Whenever $J\supseteq I$, the natural morphism between the groupoid presentations
  $ \left[  R_J \rightrightarrows V^a_J \right]$ and $\left[  R_I \rightrightarrows V^a_I \right]$  induces the  morphism of stacks $\phi_J^I: Y_J \to Y_I$. In particular, $\phi_I^I=\mbox{id}_{Y_I}$. The space  $Y_{\emptyset}=X$ will be called the target of the network.

\end{definition}

\begin{definition}
Consider a proper local embedding of Noetherian stacks
$g: Y \to X$,  with $Y$ integral. If $g$ factors through  an \'etale epimorphism $e: Y\to D_{Y/X}$ and a proper local embedding $g_1: D_{Y/X}\to X$ of generic degree 1, then we define $Y_I:=D_{Y/X, I}\times_{D_{Y/X}}Y$, for the network consisting of $\{ D_{Y/X, I},  \varphi_J^I\}_{I\subseteq J\not=\emptyset}$ constructed as in the preceding definition,  a target  $Y_{\emptyset}=X$ and the morphisms $g_i:Y_i\to X$. The morphisms  $\phi_J^I: Y_J \to Y_I$ are also obtained by pull-back from the network of $D_{Y/X}$.
\end{definition}


\begin{remark} Even though each space $Y_I$ in the network of $g$ and $U$ is intrinsic to the morphism $g$ (\cite{noi1}, Corollary 2.8), the network itself depends on the choice of the suitable cover $U$,  inasmuch as the number of copies of the same space $Y_I$ can vary from network to network. For example, if we replace $U$ by a disjoint union of $m$ copies of $U$, where $m$ is a positive integer, then the network of $g, U$ is replaced by $m$ of its copies, with the exception of the final target $X$ which is unique. In the next proposition we will show that there is, however, a canonical choice of a minimal network for the morphism $g$, which will make the subsequent construction of an \'etale lift of $g$ canonical, too.

\end{remark}

\begin{notation} Consider now a  proper local embedding of Noetherian stacks $g:Y\to X$ of generic degree one, with $Y$ integral.
 For every natural number $n$, we denote by $\prod_X^n Y$ the fibered product  over $X$ of $n$ copies of $Y$. We denote by $\Delta_n$ the union of the images of all diagonal morphisms $\prod_X^m Y \to \prod_X^n Y $ for $m\leq n$, and by $Y^n$ the complement of $\Delta_n$ in $\prod_X^n Y$.
 \end{notation}

 \begin{lemma}
 $Y^n$ is a closed substack of $\prod_X^n Y$.
 \end{lemma}

\begin{proof}
We only need to check that the image of the diagonal morphism $Y\to Y\times_XY$ is both open and closed in $Y\times_XY$. Then, by  induction on $n$  we obtain that $\Delta_n$ is a union of connected components of  $\prod_X^n Y$ for any $n>1$.
Indeed, since $g:Y\to X$, then so is $Y\to Y\times_XY$. On the other hand, to prove that the image of this morphism is open,
we choose any cover $U$ of $X$ suitable for the morphism $g$. Let $V=\bigsqcup_iV_i:=Y\times_XU$, such that each $V_i$ is imbedded as a closed subscheme of $U$. For any indices $i,j$ as above, $V_i\times_XV_j$ is an \'etale cover of $Y\times_XY$, and $(V_i\times_XV_j)\times_{Y\times_XY}Y \cong V_i\times_YV_j$. On the other hand,
\bea \bigsqcup_j(V_i\times_YV_j) = V_i\times_YV\cong V_i\times_Y(Y\times_XU)\cong V_i\times_XU\cong \bigcup_j(V_i\times_XV_j),\eea
and so $V_i\times_XV_j \cong (V_i\times_YV_j) \bigsqcup (\bigsqcup_{k}(V_i\times_YV_k) \bigcap (V_i\times_XV_j))$.
\end{proof}

\begin{definition} \label{the canonical network}
  Let $n_g$ be the largest integer such that $Y^{n_g}$ is non-empty.  We denote by $\cN(Y/X)$ the network made out of stacks $Y_J:=Y^{|J|}$, for any $J\subseteq \{1,...,n_g \}$, and of morphisms
  $\phi_J^I: Y_J \to Y_I$, defined by restrictions of the natural projections, for $I\subseteq J$. Here $Y_{\emptyset}=X$ and $\phi_i^{\emptyset}=g$ for any $i \in \{1,...,n_g \}$.
  For a general proper local embedding $g$, let $\cN(Y/X):=\cN(D_{Y/X}/X)\times_{D_{Y/X}}Y$.

   $\cN(Y/X)$ will be called the canonical network of the the finite local embedding $g:Y \to X$.
\end{definition}

\begin{proposition} \label{properties of the canonical network}

a) There exists an \'etale cover $U$ of $X$ suitable for $g$ such that $\cN(Y/X)$ is the network of local embeddings associated to $g, U$.

b) If $X'\to X$ is a morphism and $Y'\cong Y\times_XX'$, then $\cN(Y'/X')\cong \cN(Y/X)\times_XX'$.
\end{proposition}
\begin{proof}
 Consider any   \'etale covering $U'$ of $X$ suitable for $g$.  At least one such cover exists, by Proposition 1.9 in \cite{noi1}.   Let  $\phi'^{I'}_{J'}: Y_{J'} \to Y_{I'}$ denote the morphisms in the associated network. By examining the respective groupoid presentations it can be proven (\cite{noi1}, Corollary 2.8.) that the spaces $ Y_{J'}$ are isomorphic to $Y^{n}$ and, moreover, by the same proof, the morphisms are restrictions of projections as above. It remains to check that, after possibly "pruning" $U'$ , the associated network has the required set of nodes and morphisms. Indeed, assume that $g(Y)\times_XU'=\bigcup_{l\in \{1,...,m\}} W_l$, with $W_l$ as in Definition \ref{U}. If $m>n_g$, let $U:=U'\setminus (\bigcup_{l=n_g+1}^m W_l)$. The induced map $U \to X$ is \'etale and also surjective, due to the maximality of $n_g$ and to property (2) in Definition \ref{U}.
As $W_I \cong W_{I'}$ whenever $|I|=|I'|$, then the network associated to $U$ has exactly the right number of nodes and morphisms as  $\cN(Y/X)$.
The second statement is due to the definition of $\cN(Y/X)$ and Proposition \ref{split into etale and generically deg. 1}.
\end{proof}

\begin{lemma} \label{subnetworks}
Let $g:Y \to X$ be a finite local embedding of Noetherian stacks, and let $\cN(Y/X)=\{ \phi_J^I: Y_J \to Y_I \}_{I\subseteq J\subseteq \{1, ..., n_g\} }$ be its canonical network. Then for any integer $k$ with $0\leq k< n_g$, the projection morphism
 and for any $K\subseteq L \subseteq \{1, ..., n_g\}$ with $|K|=k$ and $|L|=k+1$,  the morphism $\phi^K_L: Y_{L} \to Y_{K}$ is a finite local embedding with associated canonical network
\bea   \cN(Y_L/Y_{K})=\{ \phi_J^I: Y_J \to Y_I \}_{K\subseteq I\subseteq J\subseteq \{1, ..., n_g\} }.    \eea
Here by convention $Y_0=X$.
\end{lemma}
\begin{proof}
The lemma is due to the existence of canonical isomorphisms $\prod^l_{\prod^k_XY}\prod^{k+1}_XY \cong \prod^{l+k}_XY$, which commute with the projections $\prod^l_{\prod^k_XY}\prod^{k+1}_XY \to \prod^{l-1}l_{\prod^k_XY}\prod^{k+1}_XY$ and $\prod^{l+k}_XY\to \prod^{l+k-1}_XY$, respectively, and with the respective diagonal morphisms.

\end{proof}

\begin{definition}\label{define all networks}
Consider a network $\cN$ of proper local embeddings $\phi_J^I: Y_J \to Y_I$ for
$I\subseteq J\in \cP$ associated to a
proper local embedding $g: Y\to X$, where by
convention $Y_{\emptyset }=X$. We will briefly describe here an \'etale lift $\cN^0$ of $\cN$ which is a configuration stack, namely a network of closed embeddings
$\cN^0=  \{  \phi_J^{I 0}: Y_J^0 \hookrightarrow Y_I^0   \}$, and a morphism $p^0:\cN^0 \to \cN$ which is \'etale, in the sense that each of the constituent morphisms is \'etale. A detailed proof of the existence of this network based on \'etale coverings can be found in \cite{noi1}, (Theorem 1.5).
  $\cN^0$ is in fact the last of a sequence of networks $\{\cN^i\}_{n_g\geq i\geq 0}$ constructed inductively, where $\cN^{n_g}:=\cN$, and for each index $i$,
  \begin{enumerate}
   \item the morphisms  $\phi_J^{I i}: Y_J^i \hookrightarrow Y_I^i$  are closed embeddings for all $I, J$ such that $ J\supseteq I\in \cP_k$ with $k\leq i$;
  \item there is an \'etale morphism $\cN^{i-1} \to \cN^i$.
  \end{enumerate}
The sequence is constructed as follows. Assume that $\cN^i$ with the property $(2)$ above has been defined. For each $I\in \cP$, we denote by $S_I^i$ the stack obtained by gluing all stacks $Y^i_J$ satisfying $J\supset I$ like in Lemma \ref{gluing in a network}. Each $S_I^i$ comes with a map $S_I^i \to Y_I^i$ which is in fact proper and \'etale on its image, and the set of all these maps defines a map of networks $\cS^i \to \cN^i$. With the notations from Proposition \ref{lift for morph etale on its image}, we define  $\cN^{i-1} := F_{\cS^i/\cN^i}$ in the obvious sense: each $Y_I^{i-1}= F_{S_I^i/Y_I^i}$, the \'etale lift of $S_I^i\to Y_I^i$. Thus there exists a natural \'etale map $\cN^{i-1} \to \cN^i$. We note that for $I\in \cP_k$ with $k\geq i$, the morphism  $S_I^i\to Y_I^i$ is already a closed embedding, so $Y_I^{i-1}=Y_I^i$, while for $I \in \cP_{i-1}$ and $J \supseteq I$, we have $Y_J^{i-1}=Y_J^i \hookrightarrow S_I^{i} \hookrightarrow Y_I^{i-1}$, a closed embedding.

\end{definition}

\begin{definition} Consider a proper local embedding $g: Y\to X$. Let $\{Y^{a}\}_a$ denote the set of irreducible components of $Y$, and for each $a$ let $\cN^i(Y^{a}/X)=\{  \phi_J^{a, I, i}: Y_J^{a,i} \hookrightarrow Y_I^{a,i} \}_{J\supseteq I; I, J\in \cP(\Lambda^a)}$ denote the networks associated to the restriction of $g$ on $Y^a$, as in the previous definition. Here $0\leq i\leq |\Lambda^a|$, and $|\Lambda^a|$ is the largest number such that $(Y^{a})^n$ is nonempty for all $n\leq |\Lambda^a|$. We define the canonical network of the local embedding $g$ by
\bea \cN(Y/X) = \cN^{\sum_a|\Lambda^a|}(Y/X) := \times_{X }\{\cN(Y^{a}/X)\}_a= \times_{X }\{ \cN^{|\Lambda^a|}(Y^{a}/X)\}_a.  \eea
The objects of this network are fibered products over $X$ of factors $Y^{a}_{I^a}$ for all $a$, where $I^a \subseteq \Lambda^a$.
All the networks $\cN^{i}(Y/X)$ for $0\leq i\leq |\Lambda^a|$ are constructed inductively by the process outlined in the previous definition.
\end{definition}

\begin{proposition} \label{reducible}
Consider a proper local embedding $g: Y\to X$, and let $\{Y^{a}\}_a$ denote the set of irreducible components of $Y$. With the notations from the previous definitions,
\bea \cN^{0}(Y/X) \cong \times_{X }\{\cN^0(Y^{a}/X)\}_a,  \eea
the fiber product over $X$ of all the networks $\cN^0(Y^{a}/X)$.
\end{proposition}

\begin{proof}
The proof relies on induction after the number of irreducible components, as well as decreasing induction after the step $i$ in the construction of the networks $\cN^{i}(Y/X)$, and largely on property (10) in Proposition \ref{lift for morph etale on its image}. The induction after the number of irreducible components reduces to proving the proposition for $Y=Z \bigcup T$. Denote
\bea \cN(Y/X) = \cN^{m+n}(Y/X) := \cN(Z/X)\times_X\cN(T/X)= \cN^{m}(Z/X)\times_X\cN^{n}(T/X),  \eea
and $\cN^i(Z/X)= \{  \phi_J^{ I,i}: Z_J^{i} \hookrightarrow Z_I^{i} \}_{J\supseteq I; I, J\in \cP(\Lambda)}$, with $|\Lambda|=m$, while $\cN^j(Z/X)= \{  \phi_B^{ A,j}: T_B^{j} \hookrightarrow T_A^{j} \}_{B\supseteq A; A, B\in \cP(\Gamma)}$ with $|\Gamma|=n$. Our induction hypothesis will be that for a fixed $k$ integer, $0\leq k\leq m+n$, the objects of the network $\cN^k(Y/X)$ are of the form
\begin{itemize}
\item[(1)] $Z^{|I|}_I\times_XT^{|A|}_A$, if $|I|+|A|\geq k$, and
\item[(2)] $F_{S_{I\cup A}^{k+1}/(Z_I\times_XT_A)}$ otherwise.
\end{itemize}
Here the notations are consistent with Definition \ref{define all networks}, and the index $(k+1)$ refers to the naturally corresponding objects in the $k+1$-th network of $g:Y\to X$. Thus (2) is a direct consequence of the definition of the objects in $\cN^{k}(Y/X)$, together with property (8) in Proposition \ref{lift for morph etale on its image} applied successively to the compositions $S_{I\cup A}^{l+1} \hookrightarrow S_{I\cup A}^{l} \to F_{S_{I\cup A}^{l+1}/Z_I\times_XT_A}$ for $l\geq k+1> |I|+|A|$ .
Condition (1) is  clearly satisfied when $k=m+n$. If satisfied for a fixed $k$, then for any $I\in \cP(\Lambda) $ and  $A\in \cP(\Gamma)$  there is a Cartesian diagram
\bea \diagram  S^{|I|+1}_I\times_XS^{|A|+1}_A \rto \dto & S^{|I|+1}_I\times_X T^{}_A \dto \\
Y^{}_I\times_X S^{|A|+1}_A \rto & Y^{}_I\times_X T^{}_A.  \enddiagram \eea
whenever $|I|+|A|=  k-1$.
Gluing in the network  $\cN^k(Y/X)$ yields
\bea  & S_{I\bigcup A}^k=  (S^{|I|+1}_I\times_X T^{|A|}_A)\bigcup_{S^{|I|+1}_I\times_XS^{|A|+1}_A}(Y^{|I|}_I\times_X S^{|A|+1}_A)= & \\& =( F_{S^{|I|+1}_I\times_XS^{|A|+1}_A/S^{|I|+1}_I\times_X T^{}_A})\bigcup_{S^{|I|+1}_I\times_XS^{|A|+1}_A}(F_{S^{|I|+1}_I\times_XS^{|A|+1}_A/Y^{}_I\times_X S^{|A|+1}_A}),& \eea
(in accord with property (8), Proposition \ref{lift for morph etale on its image}).
Thus by (2) above and properties (10), (8) in Proposition \ref{lift for morph etale on its image}, when $|I|+|A|=  k-1$,
\bea  & (Y_I\times_X T_A)^{k-1}\cong F_{S_{I\bigcup A}^k/Y_I\times_X T_A}\cong & \\& \cong ( F_{S^{|I|+1}_I\times_X T^{}_A/Y^{}_I\times_X T^{}_A}) \times_{Y^{}_I\times_X T^{}_A} (F_{Y^{}_I\times_X S^{|A|+1}_A/Y^{}_I\times_X T^{}_A}) \cong & \\
& \cong   F_{S^{|I|+1}_I/Y^{}_I} \times_X F_{S^{|A|+1}_A/T^{}_A} \cong  Y^{|I|}_I\times_X T^{|A|}_A.&   \eea
(Here $F_{S^{|I|+1}_I/Y^{}_I} \cong  Y^{|I|}_I$  due to property (8) in Proposition \ref{lift for morph etale on its image}, applied successively to the compositions
 $S^{l+1}_I \hookrightarrow S^{l}_I \to Y^{l}_I$ for $l> |I|$.) This ends the proof of the induction step.

\end{proof}

\begin{definition} \label{the functor F_{Y/X}} Let $g: Y \to X$ be a proper local embedding. For each positive integer $k$,  let $Y^k_X$ denote the complement of all the diagonals in the $k$-th fibered product of $Y$ over $X$.
Let $n$ be the largest integer such that $Y^n_X$ is non-empty.
   We define the functor $F_{Y/X}: \Sch_{/X} \to \Sets$ as follows: For any scheme $T$ and any morphism $T \to X$ given by an object $\alpha \in X(T)$, we consider the set of
   all tuples $(( T_i, \beta_i, f_i)_i)_{i\in \{ 1,...,n\} }$, where
 \begin{itemize}
\item[(1)]   $T_i$ are closed subschemes of $T$ such that for $I \subseteq \{ 1,..., n\}$, the intersections $T_I=\bigcap_{i\in I}T_i$ (where by convention $T_{\emptyset}=T$) satisfy
   \bea T_I\times_Xg(Y^k_X) = \bigcup_{J\supseteq I; |J|=k+|I|}T_J,\eea
 \item[(2)]  $\beta_i \in Y(T_i)$ are objects whose pullbacks to any of the subsets $T_I$ are pairwise distinct (non-isomorphic), and
  \item[(3)] $f_i$ is an isomorphism between $g(\beta_i)$ and $\alpha_{|T_i}$.
    \end{itemize}
   \end{definition}

\begin{theorem} \label{main}
Let $g: Y \to X$ be a proper local embedding. The functor $F_{Y/X}$ is a stack. Moreover, there exists a unique morphism $F_{Y/X}\to X$, \'etale and universally closed, with the following properties:
\begin{enumerate}
\item $F_{Y/X} \times_Xg(Y) \cong S_{Y/X},$ where
$S_{Y/X}=S_{\{1,2,...,n\}}$ is the stack constructed by gluing the stacks $\{F_{Y_{ij}/Y_{i}}\}_{i\not=j; i,j\in \{1,...,n\}}$ within the network $\cN^0(Y/X)$. Furthermore, $F_{Y/X} \setminus S_{Y/X} \cong X\setminus Y$, and the \'etale morphism $F_{Y/X}\to X$ is uniquely (up to a unique isomorphism) defined by these properties.
\item  For each object $Y_{I}$ in $\cN(Y/X)$,
\bea Y_{I}\times_XF_{Y/X} \cong \bigsqcup_{|I_0|=|I|}F_{Y_{I_1}/Y_{I_0}},\eea
where $I_1\supset I_0$ is a fixed choice such that $|I_1|=|I_0|+1$.
\item If $g:Y\to X$ is a closed embedding, then $F_{Y/X}\cong X$.
\item If $g:Y\to X$ is \'etale and proper, and $X$ is connected then $F_{Y/X}\cong Y$.
\item For any morphism of stacks $u: X' \to X$ and $Y':= Y\times_{X}X'$,
there exists a morphism $F_u:F_{Y'/X'} \to F_{Y/X}$ making the squares in the following diagram Cartesian:
\bea     \diagram    Y' \dto \rto & {F_{Y'/X'}}  \rto \dto^{F_u} & {X'} \dto^{u}  \\
Y \rto & {F_{Y/X}} \rto & {X}. \enddiagram  \eea
\item If $h: Z \to Y$ is proper and \'etale on its image, and $g:Y\to X$ is a closed embedding, then $F_{Z/Y}\cong Y\times_XF_{Z/X}$. In particular, there exists a natural \'etale morphism $g_*:F_{Z/Y} \to F_{Z/X}$.
\item If $g: Y\bigcup T \to X$ is a local embedding, then
\bea F_{Y/X}\times_XF_{T/X}  \cong F_{Y\bigcup T/X}.\eea

\end{enumerate}

\end{theorem}

\begin{proof}

With the notations from Definitions \ref{the canonical network}  and \ref{define all networks}, we will show that $F_{Y/X}= X^0$, the target of the network $\cN^0(Y/X)$. More generally, if $I, J\subset \{ 1,...,n\}$ are such that $J\supset I$ and $|J|=|I|+1$, then we will show by decreasing induction on $I$ that
$F_{Y_{J}/Y_{I}}= Y_{I}^{|I|}$, the target of the network  $\cN^0(Y_{J}/Y_{I})$. We recall that $\cN^0(Y_{J}/Y_{I})$ is a subnetwork of $\cN^0(Y/X)$, due to Lemma \ref{subnetworks} and Definition \ref{define all networks}.

As a first step, we notice that when $|I|=n-1$, the functor $F_{Y_{J}/Y_{I}}$, coincides with the one in Proposition \ref{lift for morph etale on its image}, and it is thus a stack satisfying all the required properties.
 Consider now a general $I$ and assume that   $F_{Y_{K}/Y_{J}}= Y_{J}^{|J|}$ for all $K\supset J\supset I$ with $|K|=|J|+1=|I|+2$. Consider now a scheme $T_I$ with a morphism $T_I\to Y_{I}$ and a set of data  $( T_{I \bigcup\{ i \}}, \beta_{I \bigcup\{ i \}}, f_{I \bigcup\{ i \}})_{i\not\in I}$ like in Definition \ref{the functor F_{Y/X}}. Thus
 \bean \label{fiber prod1} T_I \times_{Y_I} \phi_J^I(Y_J) \cong \bigcup_{i\not\in I}T_{I \bigcup\{ i \}}. \eean  Moreover, for each $J$ as above, the set of data consisting in
   \bea  \beta_{J} \mbox{ together with the tuple} ( T_{J \bigcup\{ i \}}, \phi_{J \bigcup\{ i \}}^{I \bigcup\{ i \} *}\beta_{I \bigcup\{ i \}}, \phi_{J \bigcup\{ i \}}^{I \bigcup\{ i \} *}f_{I \bigcup\{ i \}})_{i\not\in J}\eea determine an element in  $F_{Y_{K}/Y_{J}}(T)$ and thus by the induction hypothesis, a morphism $T_J \to Y_{J}^{|J|}$, making the following diagrams commutative:
 \bea   \diagram T_{J \bigcup\{ i \}} \dto \rto &  Y_{J \bigcup\{ i \}}^{|J|} \dto \\
 T_J \rto  & Y_{J}^{|J|}, \enddiagram \eea
 for all $i\not\in J$.
 Composition with the closed embeddings $ Y_{J}^{|J|} \hookrightarrow S_{I}^{|J|}$, for  $S_I^{|J|}$ like in Definition \ref{define all networks},
 yields maps $T_J \to S_{I}^{|J|}$ for all $J$ as above, which glue to $\bigcup_{i\not\in I}T_{I \bigcup\{ i \}}  \to S_{I}^{|J|}$. This, together with equation (\ref{fiber prod1}) and Proposition \ref{lift for morph etale on its image}, insure the existence of a natural morphism $T_I \to Y_{I}^{|I|}$ compatible with the data $( T_{I \bigcup\{ i \}}, \beta_{I \bigcup\{ i \}}, f_{I \bigcup\{ i \}})_{i\not\in I}$. This ends the induction step.

 From here, properties (1) and (2), and (6) follow from the construction of the network $\cN(Y/X)$, together with Proposition \ref{lift for morph etale on its image}. Properties (2) and (3) are direct consequences of (1). Property (4) also follows the construction of the network $\cN(Y/X)$, together with properties listed in Proposition \ref{properties of the canonical network} b), Lemma \ref{gluing in a network}, b) and Proposition \ref{properties of split}, part (1). Property (7) is a consequence of Proposition \ref{reducible}.

\end{proof}

\begin{example}
 If $g: Y \to X$ is proper and \'etale on its image, then $F_{Y/X}$ coincides with the stack defined in Proposition \ref{lift for morph etale on its image}.

 \end{example}

\begin{example} For a separated Deligne-Mumford stack $X$, the diagonal morphism $ \Delta: X \to X\times X$ is a
finite local embedding. Then $X \times_{X\times X} X$ is the inertia
stack $I^1(X)$ of $X$, representing objects of $X$ with their
isomorphisms. Similarly, the higher inertia stack $I^n(X)$ is
defined as the  $n$-th order product of $X$ over $X\times X$.  With
notations from \ref{the canonical network}, the objects of the
canonical stack of $ \Delta: X \to X\times X$ are the components
$I^n_0(X)$ of the inertia stacks obtained after removing all the
previous components which are images of $I^k(X)$ for $k<n$, as well
as $X$ itself, through diagonal morphisms.

\end{example}



\providecommand{\bysame}{\leavevmode\hbox to3em{\hrulefill}\thinspace}

\end{document}